\documentclass[12pt]{amsart}
 
\usepackage[T1]{fontenc}
\usepackage{amssymb,bm,mathrsfs,amsmath}
\usepackage{mathtools}
 
\usepackage[utf8]{inputenc}
 
\usepackage[colorlinks=true, linkcolor=blue, citecolor=red,
urlcolor=blue]{hyperref}

\calclayout

\title{Sphere Packings in Higher Dimension \\ (after Boaz Klartag)}
 
\author{Guillaume Aubrun}
\address{Institut Camille Jordan, Université Lyon 1}
\email{aubrun@math.univ-lyon1.fr}
 
\newcommand{\R}{\mathbf{R}}

\newcommand{\Z}{\mathbf{Z}}
\newcommand{\E}{\mathbf{E}}
\newcommand{\e}{\varepsilon}
\renewcommand{\P}{\mathbf{P}}
\newcommand{\iy}{\infty}
\newcommand{\Id}{\mathrm{Id}}
\newcommand{\Sym}{\mathrm{Sym}}
\newcommand{\gGL}{\mathsf{GL}}
\newcommand{\gSL}{\mathsf{SL}}

\DeclareMathOperator{\tr}{Tr}

\DeclareMathOperator{\vol}{vol}
\DeclareMathOperator{\covol}{covol}
\newcommand{\scalar}[2]{\langle #1 , #2\rangle}

\renewcommand{\leq}{\leqslant}
\renewcommand{\geq}{\geqslant}

\theoremstyle{plain}
\newtheorem{theorem}{Theorem}

\newtheorem{proposition}{Proposition}

\newtheorem{lemma}[theorem]{Lemma}

\theoremstyle{remark}
\newtheorem{remark}{Remark}

\begin{document}
 
\begin{abstract}
Let $\delta_n^L$ be the maximal density of a lattice sphere packing in the $n$-dimensional Euclidean space. We explain how Boaz Klartag proved the inequality $\delta_n^L \geq c n^2 2^{-n}$ where $c>0$ is a universal constant. In higher dimension, even for non-lattice sphere packings, this new lower bound is a substantial improvement.

Klartag's proof uses the probabilistic method in two different ways. The first, very standard, relies on the statistical properties of a uniformly chosen random lattice. The second, completely new, studies the stochastic evolution of an ellipsoid constrained to contain non nonzero lattice points in the interior.
\end{abstract}
 
\maketitle

This text is an English translation of the notes prepared (in French) for the Bourbaki seminar given by the author in June 2026.
 
\section{Introduction}
 
In the Euclidean space $(\R^n,\|\cdot\|)$, we denote by $B(x,r)$ the open ball of center~$x$ and radius $r>0$. A \emph{sphere packing} is a family $(B(x_i,r))$ of balls of the same radius $r$ that are pairwise disjoint, which amounts to saying that their centers satisfy $\|x_i-x_j\| \geq 2r$ for $i \neq j$. The \emph{density} of such a packing is defined, when it exists, as the limit
\[ \lim_{R \to \iy} \frac{\vol \left( B(0,R) \cap \bigcup_i B(x_i,r) \right)}{\vol ( B(0,R) )} \]
Let $\delta_n$ denote the maximum density of a sphere packing in $\R^n$. Since the problem is invariant under scaling, we may assume the spheres have radius $1$. 
Here is the exhaustive list of values known to date.
\begin{itemize}
\item $\delta_1=1$. Trivial.
\item $\delta_2 = \frac{\pi}{\sqrt{12}}$. This density is achieved for balls centered at the points of the hexagonal lattice.
\item $\delta_3 = \frac{\pi}{\sqrt{18}}$. This is Kepler's conjecture proved by \cite{Hales05}. This density is achieved for balls centered at the points of the face-centered cubic lattice.
\item $\delta_8 = \frac{\pi^4}{384}$ \cite{Viazovska17}. This density is achieved for balls centered at the points of the $E_8$ lattice.
\item $\delta_{24} = \frac{\pi^{12}}{12!}$ \cite{CKMRV17}. This density is achieved for balls centered at the points of the Leech lattice.
\end{itemize}
 
Let $\delta_n^L$ denote the maximum density of a lattice sphere packing, i.e., a sphere packing whose sphere centers form a lattice in $\R^n$. Computing $\delta_n^L$ is a fundamental question in the geometry of numbers. The value is known for $n \leq 9$ and $n=24$. A reference on these questions is \cite{Splag}; the case $n=9$ was recently resolved in \cite{SW25}.
 
It is obvious that $\delta_n^L \leq \delta_n$. In the rare cases where the value of $\delta_n$ is known, it equals $\delta_n^L$. However, it is believed that the values of $\delta_n^L$ and~$\delta_n$ differ for all sufficiently large~$n$.
 
This paper is concerned with the estimation of $\delta_n$ and $\delta_n^L$ as the dimension $n$ tends to infinity. Let us first recall the history of these questions.
 
It is very easy to prove the inequality $\delta_n \geq 2^{-n}$ by considering a maximal family of pairwise disjoint balls of radius $r$ and noting that balls with the same centers and radius $2r$ cover~$\R^n$. The inequality $\delta_n^L \geq 2^{-n}$ is more delicate; it is a consequence of the Minkowski--Hlawka theorem, stated by Minkowski as early as 1893 and proved in \cite{Hlawka43}. We will give a proof in Section~\ref{section:rogers} using \emph{random} lattices. For large $n$, no explicit example is known of a lattice sphere packing in $\R^n$ with density greater than $2^{-n}$.
 
Let us now mention the known upper bounds. The very short paper \cite{Blichfeldt29} proves the inequality
\[ \delta_n \leq \frac{n+2}{2} \Bigl( \frac{1}{\sqrt{2}} \Bigr)^n .\]
This result was substantially improved by \cite{KL78}
\begin{equation} \label{eq:KL}  \delta_n \leq ( \alpha+o(1)) ^n \end{equation}
where $\alpha \approx 0.66 < \frac{1}{\sqrt{2}}$ is an explicit constant. Better estimates of the $o(1)$ term were subsequently given by \cite{CZ14} and \cite{SZ24}.
 
This paper is primarily concerned with lower bounds.
The ``trivial'' bound $2^{-n}$ was first improved by \cite{Rogers47} to
\begin{equation}
\label{eq:rogers}
 \delta_n^L \geq c n  \cdot 2^{-n}. \end{equation}
Here and throughout this paper, the letters $c$ and $C$ denote unspecified positive reals that do not depend on the dimension.
Other approaches yielding better estimates for the constant~$c$ were subsequently proposed by \cite{DR47}, then \cite{Ball92} and \cite{Vance11}.
 
The next advance is due to \cite{Venkatesh13}, who showed,
by considering random lattices with algebraic symmetries, that the inequality
\[  \delta_n^L \geq c n (\log \log n) \cdot 2^{-n} \]
holds for infinitely many $n$. A different approach, combinatorial in spirit and not using lattices \cite{CJMS23} yielded for every dimension the lower bound
\[ \delta_n \geq \Bigl( \frac{1}{2} - o(1) \Bigr) n \log n \cdot 2^{-n} .\]
A better result, which we present here in detail, is the following.
\begin{theorem}[\cite{Klartag26}] \label{theo:main}
There exists a constant $c>0$ such that for every integer $n \geq 1$,
\[ \delta_n \geq \delta_n^L \geq c n^2 \cdot 2^{-n} .\]
\end{theorem}

For an AI-generated argument showing the lower bound $\delta_n^L \geq c n^2 (\log \log n) 2^{-n}$ for infinitely many $n$, see the very recent preprint \cite{AGZ26}.

This progress may seem modest given that an exponential gap remains between Klartag's lower bound and the best upper bound~\eqref{eq:KL}. Note however that \cite{Venkatesh13} conjectures that, up to logarithmic corrections, the quantity $2^n \delta_n^L$ is bounded above by a polynomial in~$n$. We now know that such a polynomial must have degree $\geq 2$.
 
Section~\ref{section:rogers} presents Rogers' argument giving the lower bound~\eqref{eq:rogers}. The idea is to construct, given a random lattice, an ellipsoid of large volume compatible with the lattice, i.e., whose interior contains no nonzero lattice point. This ellipsoid is determined by the $n$ successive minima of the lattice.
 
Klartag refines these ideas by using the probabilistic method \emph{twice}: within a random lattice, he considers a random compatible ellipsoid. The definition of this random ellipsoid relies on a stochastic process whose analysis involves all lattice points of sufficiently small norm, and not only the successive minima.
 
This Brownian exploration process can be defined in any closed convex set, which we do in Section~\ref{sec:exploration}. Klartag's proof is detailed in Section~\ref{sec:klartag}.

\section{Rogers' Argument} \label{section:rogers}
 
A \emph{lattice} is a set of the form $\Lambda = A ( \Z^n)$ for $A \in \gGL_n(\R)$. The number $\covol(\Lambda) \coloneqq |\det(A)|$ does not depend on the choice of the matrix $A$ and is called the \emph{covolume} of the lattice $\Lambda$. The set of lattices in $\R^n$ with covolume $1$ is identified with $\gSL_n(\R) / \gSL_n(\Z)$. The Haar measure on $\gSL_n(\R)$ induces on this set a finite measure (see \cite[Section~2.2]{Benoist09}) which can be normalized to have total mass $1$.
 
Thus, for every real number $v>0$, the set of lattices in $\R^n$ with covolume $v$ admits a unique probability measure invariant under the action of $\gSL_n(\R)$. Throughout this paper, the expression \emph{random lattice of covolume $v$} refers to this probability measure.
 
We will make central use of the following integration formula, due to \cite{Siegel45}. If $\Lambda \subset \R^n$ is a random lattice of covolume $v$, then for every function $f \colon  \R^n \to \R$ integrable with respect to the Lebesgue measure
\begin{equation} \label{eq:siegel}
\E \Bigl[ \sum_{x \in \Lambda^*} f(x) \Bigr] = \frac{1}{v} \int_{\R^n} f
\end{equation}
where $\Lambda^* \coloneqq \Lambda \setminus \{0\}$.
 
Let $B_n \coloneqq B(0,1)$ be the open unit ball of $\R^n$ and
$\omega_n \coloneqq \pi^{n/2} / \Gamma \left( \frac{n}{2} +1 \right)$ its volume. The \emph{minimum norm} of a lattice $\Lambda$ is the quantity $r = \min \{ \|x\| \, : \, x \in \Lambda^*\}$. The balls $\{B(x,r/2) \, : \, x \in \Lambda \}$ form a sphere packing of density $r^n \omega_n/2^n\covol(\Lambda)$.

The probabilistic method provides an immediate proof of the Minkowski--Hlawka inequality $\delta_n^L \geq 2^{-n}$. If~$\Lambda$ is a random lattice of covolume $\omega_n$, Siegel's formula~\eqref{eq:siegel} applied to the function $f = {\bf 1}_{B_n}$ gives
\[ \E \left[ |B_n \cap \Lambda^*| \right] = 1. \]
Since the number $|B_n \cap \Lambda^*|$ is even, it equals $0$ with probability $\geq \frac 12$. Thus, there exists a lattice
 $\Lambda$ in $\R^n$ with covolume $\omega_n$ and minimum norm $\geq 1$, which implies that the family $\{ B(x,\frac{1}{2}) \, : \, x \in \Lambda \}$ is a sphere packing of density $\geq 2^{-n}$.
 
We now detail Rogers' proof of the inequality $\delta_n^L \geq cn2^{-n}$. We call an \emph{ellipsoid} the image of the open unit ball under an invertible linear map. The starting point is the observation that if $\Lambda \subset \R^n$ is a lattice and $\mathcal{E}$ is an ellipsoid containing no point of $\Lambda^*$, then
\begin{equation} \label{eq:ellipsoidpacking} \delta_n^L \geq \frac{\vol (\mathcal{E})}{\covol (\Lambda)} 2^{-n}.\end{equation}
Indeed, if $\mathcal{E} = T(B_n)$ for $T \in \gGL_n(\R)$, then the family $\left\{ B(x,1) \, : \, x \in 2 T^{-1}\Lambda \right\}$ is a sphere packing of density $\omega_n/\covol(2T^{-1}\Lambda)=2^{-n}\vol(\mathcal{E})/\covol(\Lambda)$.
 
Let $\Lambda$ be a random lattice of covolume $v \coloneqq \frac{\omega_n}{n}$.
Consider the function
\[ f(x) \coloneqq \Bigl( \log \frac{1}{\|x\|} \Bigr)_+  = \Bigl( \log \frac{1}{\|x\|} \Bigr) {\bf 1}_{B_n}(x).\]
By Siegel's formula~\eqref{eq:siegel}, we have
\[ \E \Bigl[ \sum_{x \in \Lambda^*} f(x) \Bigr] = \frac 1v \int_{\R^n} f = \frac{n}{\omega_n} \int_{B_n} \log \frac{1}{\|x\|} \, \mathrm{d}x.\]
This integral is computed in polar coordinates as
\[ \frac{1}{\omega_n} \int_{B_n} \log \frac{1}{\|x\|} \, \mathrm{d}x = \frac{\int_0^1 \log(1/r)r^{n-1}\,\mathrm{d}r}{\int_0^1 r^{n-1}\,\mathrm{d}r} = \frac{1/n^2}{1/n} = 1/n.\]
It follows that
\[ \E \Bigl[ \sum_{x \in \Lambda^*} f(x) \Bigr] = 1\]
and in particular, there exists a lattice $\Lambda$ of covolume $v$ satisfying $\sum_{x \in \Lambda^*} f(x) \leq 1$. Fix such a lattice.
 
For $1 \leq i \leq n$, set
\[ \lambda_i \coloneqq  \inf \{ \lambda \, : \, \dim \left( \mathrm{Vect} ( \Lambda \cap \lambda B_n ) \right) \geq i \}. \]
The numbers $\lambda_1 \leq \cdots \leq \lambda_n$ are called the \emph{successive minima} of the lattice $\Lambda$. There exist elements $y_1,\dots,y_n$ of $\Lambda$ forming a basis of $\R^n$ and satisfying $\|y_i\|=\lambda_i$. We have
\begin{align*}
-\log (\lambda_1\cdots \lambda_n) & =  \log \frac{1}{\|y_1\|} + \dots + \log \frac{1}{\|y_n\|} \\
& \leq  f(y_1) + \dots + f(y_n) \\
& \leq  \sum_{y \in \Lambda^*} f(y) =  1.
\end{align*}
We deduce the inequality $\lambda_1 \cdots \lambda_n \geq 1/e$. Let $(e_1,\dots,e_n)$ be an orthonormal basis satisfying, for $1 \leq i \leq n$,
\[ \mathrm{Vect} (e_1,\dots,e_i) = \mathrm{Vect} (y_1,\dots,y_i) .\]
Consider the ellipsoid
\[ \mathcal{E} = \Bigl\{ \sum_{i=1}^n \alpha_i e_i \, : \, \sum_{i=1}^n \frac{\alpha_i^2}{\lambda_i^2} < 1 \Bigr\} \]
and let $x= \sum_{i=1}^n \alpha_i e_i$ be an element of $\Lambda^*$. Let $j$ be the maximal index such that $\alpha_j \neq 0$. We have $\|x\| \geq \lambda_j$, and moreover
\[ \sum_{i=1}^n \frac{\alpha_i^2}{\lambda_i^2} = \sum_{i=1}^j \frac{\alpha_i^2}{\lambda_i^2}
\geq \frac{\|x\|^2}{\lambda_j^2} \geq 1.
\]
Thus the ellipsoid $\mathcal{E}$ contains no point of $\Lambda^*$. In view of~\eqref{eq:ellipsoidpacking}, this implies the inequality
\[ \delta_n^L \geq \frac{\vol (\mathcal{E})}{v} 2^{-n}= n 2^{-n} \lambda_1 \cdots \lambda_n \geq e^{-1}n2^{-n} .\]
This completes the proof of the lower bound~\eqref{eq:rogers}.

\section{Brownian Exploration of a Closed Convex Set} \label{sec:exploration}
 
Throughout this section, we fix a Euclidean space $E$ of dimension $n$ and a closed convex subset $\mathscr{C} \subset E$ containing no affine line. We also fix a \emph{standard Brownian motion} $(B_t)_{t \geq 0}$ in $E$, i.e., a stochastic process with values in $E$ such that, for every orthonormal basis $(e_1,\dots,e_n)$ of $E$, the $n$ scalar stochastic processes $(\scalar{B_t}{e_i})_{t \geq 0}$, for $1 \leq i \leq n$, are independent Brownian motions. Let $(\mathcal{F}_t)_{t \geq 0}$ be the canonical (complete) filtration of $(B_t)$. That is, for $t \geq 0$, $\mathcal{F}_t$ is the smallest $\sigma$-algebra containing all negligible sets and making measurable the random variables $(B_s)_{0 \leq s \leq t}$.
 
\medskip
 
We will define a stochastic process with values in $\mathscr{C}$ which we call the \emph{Brownian exploration} of $\mathscr{C}$. Informally, the Brownian exploration of $\mathscr{C}$ is constructed by conditioning the Brownian motion $(B_t)$ to belong to a face of $\mathscr{C}$ reached at an earlier time. This process is a martingale that almost surely reaches an extreme point of $\mathscr{C}$ in finite time.
 
\medskip
 
For every $x \in \mathscr{C}$, the set
\begin{equation}
\label{eq:defVxC}
V(x,\mathscr{C}) \coloneqq \{ y \in E : \exists \e > 0 \ : \ |t| < \e \Longrightarrow x+ty \in \mathscr{C} \} \end{equation}
is a vector subspace of $E$. The equality $V(x,\mathscr{C})=E$ holds if and only if $x$ is an interior point of $\mathscr{C}$, and the equality $V(x,\mathscr{C})=\{0\}$ holds if and only if $x$ is an extreme point of $\mathscr{C}$. Within the vector subspace $V(x,\mathscr{C})$, the point $0$ is interior to the convex set $(\mathscr{C}-x) \cap V(x,\mathscr{C})$.
 
\medskip
 
\begin{proposition} \label{prop:exploration}
Let $x_0$ be an interior point of $\mathscr{C}$. There exists a continuous martingale~$(X_t)$ adapted to the filtration $(\mathcal{F}_t)$ and satisfying the stochastic differential equation
\begin{equation} \label{eq:EPDS} X_0 = x_0,  \ \ \ \mathrm{d}X_t = p_t(\mathrm{d}B_t)\end{equation}
where $p_t$ denotes the orthogonal projection onto the subspace $V(X_t,\mathscr{C})$, with the following properties: almost surely,
\begin{enumerate}
\item[(a)] for every $t \geq 0$, we have $X_t \in \mathscr{C}$,
\item[(b)] there exist an extreme point $z$ of $\mathscr{C}$ and a real $T >0$ \textup{(}both random\textup{)} such that $X_t=z$ for all $t \geq T$,
\item[(c)] for all $0 \leq s \leq t$, we have  $V(X_t,\mathscr{C}) \subset V(X_s,\mathscr{C})$.
\end{enumerate}
\end{proposition}
 
This process $(X_t)$ is called the \emph{Brownian exploration process of $\mathscr{C}$ starting from $x_0$}. Its construction uses the following lemma recursively.
 
\begin{lemma} \label{lemme:hahn-banach}
Let~$V$ be a Euclidean space and $K \subset V$ a closed convex set, distinct from~$V$ and containing~$0$ in its interior. Let $(\beta_t)_{t \geq 0}$ be a standard Brownian motion in~$V$. The random variable
\[ \tau \coloneqq \inf \{ t \geq 0 \, : \, \beta_t \not \in K \} \]
satisfies $0 < \tau < \infty$ almost surely. Moreover, almost surely, $\beta_\tau \in \partial K$.
\end{lemma}
 
\begin{proof}
Since $K \neq V$, by the Hahn--Banach theorem, there exists an element $z \in V$ of norm~$1$ such that the linear form $\scalar{z}{\cdot}$ is bounded below on $K$. The scalar stochastic process $(\scalar{z}{\beta_t})_{t \geq 0}$ is a Brownian motion and therefore satisfies almost surely $\inf \{ \scalar{z}{\beta_t} \, : \, t \geq 0 \}= - \iy$ (see \cite[Corollary 2.3]{LeGall13}), which implies the existence of a real $t$ such that $\beta_t \not \in K$. We have thus shown that $\tau$ is almost surely finite. The remaining properties follow from the fact that the function $t \mapsto \beta_t$ is almost surely continuous.
\end{proof}
 
\begin{proof}[Proof of Proposition~\ref{prop:exploration}]
Set $X_0=x_0$. To construct $(X_t)_{t \geq 0}$, define a sequence of stopping times
\[ 0 = \tau_0 < \tau_1 < \dots < \tau_n \]
such that if $t \in [\tau_i,\tau_{i+1}[$ then $V(X_t,\mathscr{C}) = V_i$, where $V_i \coloneqq V(X_{\tau_{i}},\mathscr{C})$.
Let~$\pi_i$ denote the orthogonal projection onto $V_i$. We define $(X_t)$ so that, for $t \in [\tau_i,\tau_{i+1}[$,
\[ X_t = X_{\tau_i}+ \pi_i(B_t-B_{\tau_{i}}) .\]
 
The construction proceeds by induction on $i \in \{0,\dots,n\}$. Assuming $\tau_0,\dots,\tau_{i}$ and $(X_t)_{0 \leq t \leq \tau_{i}}$ have been constructed, define $\tau_{i+1}$ as follows.
\begin{itemize}
\item If $V_i \neq \{0\}$, set
\[ \tau_{i+1} = \tau_{i} + \inf \{ t > 0 \, : \, V(X_{\tau_{i}} + \pi_i(B_{\tau_i+t} - B_{\tau_{i}}),\mathscr{C}) \neq V_i) \} .\]
Apply Lemma~\ref{lemme:hahn-banach} to the Euclidean space $V_i$, the convex set $K = (\mathscr{C}-X_{\tau_i}) \cap V_i$, and the process $(\beta_t)$ defined by $\beta_t \coloneqq \pi_i(B_{\tau_i+t} - B_{\tau_{i}})$, which is a standard Brownian motion in~$V_i$. Note also that since $\dim (V_i) \geq 1$ and $\mathscr{C}$ contains no line, the set $K$ is not equal to~$V_i$. We can therefore conclude that almost surely, $\tau_{i+1} < \iy $ and $\dim (V_{i+1}) < \dim(V_i)$.
\item If $V_i= \{0\}$, the induction terminates. Set $X_t = X_{\tau_i}$ for all $t \geq \tau_i$. Also set
(arbitrarily) $\tau_{k}=\tau_j+k$ for every integer $j < k \leq n$.
\end{itemize}
 
 By construction, $(X_t)$ is a continuous martingale satisfying condition (a), condition (b) with $T=\tau_n$ and condition (c) of Proposition~\ref{prop:exploration}. 
Moreover, for any $t>0$ and $0 \leq i < n$
\[ \int_{\tau_i \wedge t}^{\tau_{i+1} \wedge t} \, \mathrm{d}X_s = 
\int_{\tau_i \wedge t}^{\tau_{i+1} \wedge t} p_s(\mathrm{d}B_s).
\]
Summing over $i$ gives
\[ \int_{0}^{\tau_{i+1} \wedge t} \, \mathrm{d}X_s = 
\int_{0}^{\tau_{i+1} \wedge t} p_s(\mathrm{d}B_s).
\]
Since on the event $s> \tau_{i+1}$ we have $\mathrm{d}X_s =0$ and $p_s=0$, we obtain
\[ X_t - X_0 = \int_{0}^t p_s( \mathrm{d}B_s ) ,\]
which is the integrated form of \eqref{eq:EPDS}.
\end{proof}
 
\begin{remark} \label{rema:C1}
For $1 \leq i \leq n$, the law of the $i$-th stopping time $\tau_i$ defined in the preceding proof is a non-atomic measure on $\R_+$. Indeed, let $\e>0$ be small enough that the ball $B(x_0,\e)$ is contained in $\mathscr{C}$, and let $\tau' \coloneqq \inf \{ t>0 \, :\, \|B_t-x_0\| \geq \e \}$ be the corresponding exit time. By the strong Markov property, the law of $\tau_i$ is the convolution of the law of $\tau'$ with the law of the $i$-th stopping time of the exploration process of $\mathscr{C}$ starting from a point chosen uniformly on the sphere $\partial B(x_0,\e)$. Since $\tau'$ has a non-atomic law \cite{CT62}, so does~$\tau_i$.
 
We will use this remark to prove Propositions~\ref{prop:Ito} and \ref{prop1}. However, it is not needed to prove Theorem~\ref{theo:main}, provided one replaces these propositions by a weakened variant where the conclusion is in integrated form, such as the one used in inequality~\eqref{eq:final}.
\end{remark}

In the remainder of this section, fix $x_0 \in \mathscr{C}$ and let $(X_t)$ denote the Brownian exploration process of $\mathscr{C}$ starting from $x_0$. As in Proposition~\ref{prop:exploration}, $p_t$ denotes the orthogonal projection onto $V(X_t,\mathscr{C})$. We will prove two coupling inequalities illustrating the idea that Brownian exploration ``diffuses more slowly'' than standard Brownian motion.
 
\begin{proposition} \label{prop:maurey}
For any norm $\|\cdot\|$ on $E$, for all $t>0$ and $\lambda > 0$,
\[ \P ( \| X_t-x_0\| \geq \lambda) \leq 2 \, \P ( \|B_t\| \geq \lambda). \]
\end{proposition}
 
\begin{proof}
Consider the process $(B'_t)$ defined by the relation $X_t = x_0 + \frac{1}{2}\left( B_t + B'_t\right)$. Since for every $t \geq 0$
\[ X_t = x_0 + \int_0^t p_s(\mathrm{d}B_s)\]
and
\[ B_t = \int_0^t \mathrm{Id}(\mathrm{d}B_s),\]
by linearity of the stochastic integral we have
\[ B'_t = \int_0^t (\mathrm{Id}-2p_s)(\mathrm{d}B_s).\]
The quadratic variation matrix of the process $(B'_t)$ equals
\[ [B']_t = \int_0^t {^t}(\mathrm{Id}-2p_s)(\mathrm{Id}-2p_s) \mathrm{d}s = t \mathrm{Id}\]
since $p_s$ is an orthogonal projection. By a theorem of L\'evy \cite[Theorem 5.4]{LeGall13}, this implies that $(B'_t)$ is a standard Brownian motion in $E$. We therefore have, writing $\|X_t-x_0\| \leq \frac{1}{2} (\|B_t\|+\|B'_t\|)$,
\begin{align*}
\P  (\| X_t-x_0\| \geq \lambda) & \leq \P  \left(\| B_t \| + \|B'_t\| \geq 2\lambda \right) \\
& \leq \P  ( \| B_t \| \geq \lambda) + \P  ( \| B'_t \| \geq \lambda)
\end{align*}
and the result follows.
\end{proof}
 
\begin{proposition} \label{prop:couplage}
Let $z \in  E$ and $a$ be a real number.
Let
\[ \tau_1 \coloneqq \inf \{ t \geq 0 \, : \, \scalar{z}{X_t} = a \} \]
and
\[ \tau_2 \coloneqq \inf \{ t \geq 0 \, : \, \scalar{z}{x_0+B_t} = a \}. \]
Then, for every real $T > 0$,
\[ \P(\tau_1 \geq T) \geq \P (\tau_2 \geq T) .\]
\end{proposition}
 
\begin{proof}
Assume without loss of generality that $\|z\|=1$, $x_0=0$ and $a>0$. The scalar stochastic process $(M_t)$ defined by $M_t \coloneqq \scalar{z}{X_t}$ is a continuous martingale. Its quadratic variation satisfies
\[ [M]_t = \int_0^t \|p_s(z)\|^2 \, \mathrm{d}s \leq
\int_0^t \|z\|^2 \, \mathrm{d}s = t .\]
The Dambis--Dubins--Schwartz theorem implies (possibly on an enlarged probability space, see \cite[Theorem 5.5]{LeGall13}) the existence of a scalar Brownian motion~$(\beta_t)$ such that
\[ M_{t} = \beta_{[M]_t}\]
for every real $t \geq 0$. Since $[M]_t \leq t$, we have
\[ \forall t < T,\ \beta_{t} < a \ \Longrightarrow \ \forall t < T, \  M_t < a . \]
The right-hand event is $\{\tau_1 \geq T\}$, and the left-hand event has the same probability as $\{\tau_2 \geq T\}$ since $(\scalar{z}{B_t})$ is a scalar Brownian motion.
\end{proof}
 
Finally, we state a consequence of It\^o's formula.
 
\begin{proposition} \label{prop:Ito}
Let $\Omega$ be an open subset of $E$ containing $\mathscr{C}$, and let $F \colon  \Omega \to \R$ be a $C^2$ function.
Suppose that for every $t>0$,
\[ (i) \int_0^t \E[ \|\nabla F(X_s)\|^2] \, \mathrm{d}s < \iy,
 \ \ \
 (ii) \int_0^t \E[ \|\nabla^2 F(X_s)\|] \, \mathrm{d}s < \iy.
\]
Then the function $g : t \mapsto \E [F(X_t)]$ is $C^1$ and satisfies for every $t \geq 0$
\[ g'(t) = \frac{1}{2} \E \left[ \tr
(\nabla^2F(X_t) p_t ) \right]. \]
\end{proposition}
 
\begin{proof}
It\^o's formula reads
\[ F(X_t) = F(x_0) + \int_0^t \scalar{\nabla F(X_s)}{p_s(\mathrm{d} B_s)} + \frac{1}{2} \int_0^t \left[ \tr (\nabla^2F(X_s) p_s )\right] \, \mathrm{d}s .\]
The process $(M_t)$ defined by $M_t \coloneqq \int_0^t \scalar{\nabla F(X_s)}{p_s(\mathrm{d} B_s)}$ is a local martingale. Hypothesis~(i), combined with the fact that $p_s$ is a contraction, implies that~$(M_t)$ is a martingale and therefore $\E[ M_t ] = 0$ for every~$t \geq 0$. Hypothesis (ii) ensures that $g(t)$ is well defined for $t \geq 0$ and equals
\[g(t) =  \frac{1}{2} \int_0^t \E \left[ \tr
(\nabla^2F(X_s) p_s ) \right]. \]
It follows from the proof of Proposition~\ref{prop:exploration} and Remark~\ref{rema:C1} that the discontinuities of the function $t \mapsto p_t$ are finite in number and their law is non-atomic. Since the function $t \mapsto \nabla^2 F (X_t)$ is almost surely continuous, this implies that the function $t \mapsto \E \left[ \tr
(\nabla^2F(X_t) p_t ) \right]$ is continuous. Thus $g$ is $C^1$ and satisfies the announced formula.
\end{proof}
 
\section{Klartag's Proof} \label{sec:klartag}
 
\subsection{Preliminaries}
 
The space $\Sym_n$ of real symmetric $n \times n$ matrices is a vector space of dimension $n(n+1)/2$. We equip it with the Euclidean structure given by the inner product $\scalar{A}{B} = \tr (AB)$. For $A$ and~$B$ in~$\Sym_n$, we write $A \preccurlyeq B$ when the matrix $B-A$ is positive semidefinite. To every symmetric positive definite matrix $A$ we associate the open ellipsoid
\[ \mathcal{E}_A \coloneqq \{ x \in \R^n \, : \, \scalar{Ax}{x} < 1 \}
 = A^{-1/2}(B_n) \]
which satisfies $\vol(\mathcal{E}_A) = \omega_n / \sqrt{\det(A)}$. Its boundary $\partial \mathcal{E}_A$ is the set of vectors $x \in \R^n$ satisfying $\scalar{Ax}{x}=1$.
 
Given a lattice $\Lambda \subset \R^n$, consider the closed convex set
\[ \mathscr{C}_\Lambda \coloneqq \{ A \in \Sym_n \, : \, \scalar{Ax}{x} \geq 1 \textnormal{ for all } x \in \Lambda^* \}.\]
 
A symmetric positive definite matrix~$A$ belongs to~$\mathscr{C}_\Lambda$ if and only if the ellipsoid~$\mathcal{E}_A$ contains no point of $\Lambda^*$. The following lemma describes the set $V(A,\mathscr{C}_\Lambda)$ defined by~\eqref{eq:defVxC}.
 
\begin{lemma} \label{lemme:CL}
Let $\Lambda$ be a lattice in $\R^n$ and $A \in \mathscr{C}_\Lambda$. Then $A$ is positive definite and satisfies the inequality
\begin{equation}
\label{eq:detlowerbound}
\sqrt{\det(A)} \geq \frac{\omega_n 2^{-n}}{ \delta_n^L \covol(\Lambda) } . \end{equation}
Moreover,
\[ V(A,\mathscr{C}_\Lambda) = \{ M \in \Sym_n \, : \,  \scalar{Mx}{x} =0 \textnormal{ for all } x \in \Lambda^* \cap \partial\mathcal{E}_A \}. \]
\end{lemma}
 
\begin{proof}
The positive semi-definiteness of $A$ follows from the fact that the set $\left\{ x/\|x\| \, : \, x \in \Lambda^* \right\}$ is dense in the unit sphere of $\R^n$. If $A$ is positive definite, inequality~\eqref{eq:ellipsoidpacking} reads
\[ \delta_n^L \geq \frac{\vol (\mathcal{E}_A)}{\covol (\Lambda)} 2^{-n} = \frac{\omega_n / \sqrt{\det(A)}}{\covol (\Lambda)} 2^{-n} \]
from which we deduce~\eqref{eq:detlowerbound}. If $A$ is only assumed positive semidefinite, the conclusion still holds by applying this argument to $A + \e \Id$ for arbitrarily small $\e>0$.
 
Let $M \in V(A,\mathscr{C}_\Lambda)$. For every $x \in \Lambda^* \cap \partial\mathcal{E}_A$, we have $\scalar{Ax}{x}=1$. The definition of $V(A,\mathscr{C}_\Lambda)$ implies $\scalar{(A+tM)x}{x} \geq 1$ for $|t|$ small enough, hence $\scalar{Mx}{x}=0$.
 
Conversely, let $M \in \Sym_n$ satisfy $\scalar{Mx}{x} = 0$ for all $x \in \Lambda^* \cap \partial\mathcal{E}_A$. Since a lattice is discrete and $A$ is positive definite, there exists $\alpha >1$ such that $\scalar{Ax}{x} \geq \alpha$ for all $x \in \Lambda^* \setminus \partial\mathcal{E}_A$. If $t$ is real with $|t|$ sufficiently small, we have $\alpha^{-1} A  \preccurlyeq A + t M$, and thus for every~$x \in \Lambda^*$
\begin{align*}
\scalar{(A+tM)x}{x}  = 1 & \ \textnormal{ if }\  x \in \partial\mathcal{E}_A  \\
\scalar{(A+tM)x}{x}  \geq \alpha^{-1} \scalar{Ax}{x} \geq 1 & \ \textnormal{ if } \ x \not \in \partial\mathcal{E}_A
\end{align*}
which implies $A+tM \in \mathscr{C}_\Lambda$. We have thus shown $M \in V(A,\mathscr{C}_\Lambda)$.
\end{proof}
 
The matrices $A \in \mathscr{C}_\Lambda$ of minimal determinant, for which equality holds in~\eqref{eq:detlowerbound}, correspond to a packing of ellipsoids $\{\mathcal{E}_A + x \, : \, x \in 2\Lambda\}$ of optimal density~$\delta_n^L$. Since the function $\log \circ \det$ is strictly concave on $\mathscr{C}_\Lambda$, these minimum points are extreme points of $\mathscr{C}_\Lambda$.
 
\subsection{Structure of the Proof}
 
Fix a standard Brownian motion $(W_t)_{t \geq 0}$ in the Euclidean space $\Sym_n$.
We will consider a Brownian exploration process constructed from $(W_t)$ as in Section~\ref{sec:exploration}.
 
Let $\Lambda \subset \R^n$ be a lattice with minimum norm $>1$, which means that the matrix~$\Id$ lies in the interior of $\mathscr{C}_\Lambda$. Let~$(A_t)$ be the Brownian exploration process of $\mathscr{C}_\Lambda$ starting from $\Id$. To simplify notation, we will pretend that the conclusions of Proposition~\ref{prop:exploration} hold everywhere on the underlying probability space; the argument becomes rigorous by adding the qualifier ``almost surely'' to each statement.
 
By Proposition~\ref{prop:exploration}, for every $t$, the matrix $A_t$ lies in $\mathscr{C}_\Lambda$. For $t$ large enough, it is an extreme point of~$\mathscr{C}_\Lambda$. The Brownian exploration process ``chooses'' an extreme sphere packing. We will see that it is possible to bound below the density of this packing, which amounts to bounding above $\det(A_t)$.
 
Since $(A_t)$ is a martingale and the function $\log \circ \det$ is concave, the function $g \colon t \mapsto \E [ \log \det A_t ]$ is decreasing. We will estimate its rate of decrease as a function of the distribution of short vectors in $\Lambda$. This requires some notation. Consider the Gaussian tail function~$\Phi$ defined for a real $a$ as
\[ \Phi(a) \coloneqq \int_a ^{\iy} e^{-x^2/2} \frac{\mathrm{d}x}{\sqrt{2\pi}} .\]
In other words, $\Phi(a)$ is the probability that a standard Gaussian random variable takes a value $\geq a$. Next, for a real $t > 0$, define a radial function $f_t \colon  \R^n \to \R$ by
\[ f_t(x) \coloneqq 2 \, \Phi \Bigl(\frac{1}{\sqrt{t}}\Bigl(1 - \frac{1}{\|x\|^2}\Bigr) \Bigr).  \]
Finally, given a lattice $\Lambda$, set
\[ K_t(\Lambda) \coloneqq \sum_{\mathclap{x \in \Lambda^* \cap D_t}}  f_t(x) \]
where $D_t \coloneqq \{ x \in \R^n \, : \, (1-3\sqrt{tn})\|x\|^2 \leq 1 \}$.
 
Set $t_{\max} \coloneqq \frac{20 \log n}{n^2}$. Theorem~\ref{theo:main} follows from the next two propositions.
 
\begin{proposition} \label{prop1}
Let $\Lambda \subset \R^n$ be a lattice with minimum norm~$>1$. Let $(A_t)$ be the Brownian exploration process of $\mathscr{C}_\Lambda$ starting from $\Id$. The function $g \colon  t \mapsto \E [\log \det A_t]$ is $C^1$ and satisfies, for $t \in [0,t_{\max}]$, the inequality
\[ g'(t) \leq  - \frac{n^2}{4} + K_t(\Lambda) + 3 \sqrt{t} n^{5/2} + \e_n \]
where $(\e_n)$ is a sequence tending to $0$ as $n \to \infty$.
\end{proposition}
 
\begin{proposition} \label{prop2}
Let $t \in [0,t_{\max}]$. If $\Lambda \subset \R^n$ is a random lattice of covolume~$\omega_n$, then
\[ \E[K_t(\Lambda)] \leq (2+\e_n) \exp \Bigl( \frac{n^2t}{8} \Bigr) \]
where $(\e_n)$ is a sequence tending to $0$ as $n \to \infty$.
\end{proposition}
 
Let us deduce Theorem~\ref{theo:main} by combining these two propositions. Let $\Lambda \subset \R^n$ be a random lattice of covolume $\omega_n$. Set $T \coloneqq \frac{16 \log n}{n^2}$. This satisfies $e^{n^2T/8}=n^2$.
 
By Proposition~\ref{prop2} (the expectation referring to the random lattice $\Lambda$),
\begin{align*} \E \Bigl[ \int_0^T K_t(\Lambda) \, \mathrm{d} t \Bigr]
& \leq  (2+\e_n) \int_0^T \exp \Bigl( \frac{n^2t}{8} \Bigr) \, \mathrm{d}t \\
& \leq  (2+\e_n) \frac{8}{n^2} \exp  \Bigl( \frac{n^2T}{8} \Bigr) = 8(2+\e_n).
\end{align*}
Take $n$ large enough so that $8(2+\e_n) < 20$. Markov's inequality gives
\[ \P \Bigl( \int_0^T K_t(\Lambda) \, \mathrm{d}t \geq 40 \Bigr) < \frac{1}{2}. \]
Recall also that a random lattice of covolume $\omega_n$ has minimum norm $>1$ with probability $\geq \frac{1}{2}$. We deduce that there exists a lattice $\Lambda$ of covolume $\omega_n$ and minimum norm $>1$ for which
\[ \int_0^T K_t(\Lambda) \,\mathrm{d}t \leq 40 .\]
Applying Proposition~\ref{prop1} to such a lattice gives (the expectation now referring to the stochastic process $(A_t)$)
\begin{align} \nonumber \E [ \log \det A_T ] = g(T) &= \int_0^T g'(t) \, \mathrm{d}t \\
\label{eq:final} & \leq  - \frac{n^2T}{4} + \int_0^T K_t(\Lambda) \, \mathrm{d}t + \int_0^T 3 \sqrt{t} n^{5/2} \, \mathrm{d}t + \e_n T \\
\nonumber & \leq  -4 \log(n) + 40 + 2 T^{3/2} n^{5/2} + \e_n T \\\nonumber & \leq   -4 \log (n) + C
\end{align}
since $T^{3/2}n^{5/2} = o(1)$. There therefore exists a matrix $A \in \mathscr{C}_\Lambda$ such that
\[ \det(A) \leq \exp( -4 \log(n) +C ) = \frac{e^C}{n^4} \]
and hence, by~\eqref{eq:ellipsoidpacking},
\[ \delta_n^L \geq \frac{2^{-n}}{\sqrt{\det(A)}} \geq c n^2 \cdot 2^{-n}.\]
 
\subsection{Proof of Proposition~\ref{prop1}}
 
We will use estimates for the operator norm (denoted $\|\cdot\|_{\mathrm{op}}$) of random matrices. Recall that $(W_t)$ is a standard Brownian motion in the Euclidean space $\Sym_n$. The law of $W_1$ is known as the Gaussian Orthogonal Ensemble (GOE). The following inequality is a consequence of the isoperimetric inequality in Gaussian space (see equation (6.38) in \cite{AS17}): for every $\e>0$,
\[ \P( \|W_1\|_{\mathrm{op}} \geq 2 \sqrt{n} +\e ) \leq \exp(-\e^2/2). \]
 
For every $t>0$, the random variable $W_t/\sqrt{t}$ has the same distribution as $W_1$. Proposition~\ref{prop:maurey} then implies, for every $\e>0$, the inequality
\begin{equation} \label{eq:subgaussian} \P \bigl( \|A_t - \Id\|_{\mathrm{op}} \geq \sqrt{t} (2\sqrt{n}+\e) \bigr)\leq 2 \exp(-\e^2/2) .\end{equation}
We will draw two consequences from this inequality. First, for every non-negative integer $k$ and every $t>0$,
\begin{equation} \label{eq:finitemoments} \int_0^t \E \left[ \|A_s\|_{\mathrm{op}}^k \right] \, \mathrm{d}s < \iy \ \textnormal{ and } \ \int_0^t \E \left[ \|A_s^{-1}\|_{\mathrm{op}}^k \right] \, \mathrm{d}s< \iy .\end{equation}
The first of these inequalities follows from~\eqref{eq:subgaussian} since sub-Gaussian random variables have all their moments finite; the second follows from the first by writing
\[ \|A_t^{-1}\|_{\mathrm{op}} \leq \frac{\|A_t\|_{\mathrm{op}}^{n-1}}{\det (A_t)} \]
and using the fact, stated in Lemma~\ref{lemme:CL}, that the determinant is bounded below by a value~$>0$ on $\mathscr{C}_\Lambda$.
 
Second, this inequality will allow us to restrict our study to the case where $A_t$ is very close to the identity matrix. For every $t > 0$, introduce the ``favorable'' event
\[ G_t = \bigl\{ \|A_t - \Id \|_{\mathrm{op}}  \leq 3\sqrt{tn} \bigr\} .\]
It follows from~\eqref{eq:subgaussian} that this event has probability very close to $1$, since the complementary event satisfies
\[ \P (\overline{G_t}) \leq 2 e^{-n/2} .\]
 
We now begin the proof of Proposition~\ref{prop1}.
Apply It\^o's formula in the form of Proposition~\ref{prop:Ito} to the function $F = \log \circ \det$. For $A \in \mathscr{C}_\Lambda$ and $B \in \Sym_n$, the Taylor expansion
\begin{align*} \log \det (A + \e B) & = \log  \det A + \log \det (\Id + \e A^{-1}B) \\
 &  = \log \det A + \e \tr(A^{-1}B) - \frac{\e^2}{2} \tr((A^{-1}B)^2) + O(\e^3)
 \end{align*}
 gives
 \begin{align*}
   \scalar{(\nabla F)(A)}{B} &= \tr (A^{-1}B) \\
 \scalar{(\nabla^2 F)(A)B}{B} &= - \tr (A^{-1}BA^{-1}B) .
 \end{align*}
 Writing $\|A\| \coloneqq (\tr A^2)^{1/2}$, we have $\| \nabla F (A)\| \leq \|A^{-1}\| \leq \sqrt{n} \|A^{-1}\|_{\mathrm{op}}$ and $\|\nabla^2 F(A)\|_{\mathrm{op}} \leq \|A^{-1}\|^2 \leq n \|A^{-1}\|^2_{\mathrm{op}}$. It follows from~\eqref{eq:finitemoments} that the hypotheses of Proposition~\ref{prop:Ito} are satisfied. Consequently, the function $g$ is $C^1$ and satisfies for $t\geq 0$
\[ g'(t) =
\frac{1}{2} \E \left[ \tr
(\nabla^2F(X_t) p_t ) \right] \]
where $p_t$ is the orthogonal projection onto the subspace $V(A_t,\mathscr{C}_\Lambda)$, whose dimension we denote $N_t$.
 
The inequality $\tr(A^{-1}BA^{-1}B) \geq \|A\|_{\mathrm{op}}^{-2} \tr(B^2)$ implies $\nabla^2F(A) \preccurlyeq - \|A\|_{\mathrm{op}}^{-2} \Id$. Thus, for $t \geq 0$,
\[
g'(t)   \leq  - \frac{1}{2} \E \left[ \|A_t\|_{\mathrm{op}}^{-2} N_t \right]  \leq - \frac{1}{2}  \E \left[ {\bf 1}_{G_t} \|A_t\|_{\mathrm{op}}^{-2} N_t \right] .
\]
If the event $G_t$ holds, then
$\|A_t\|_{\mathrm{op}} \leq 1 + 3 \sqrt{tn}$ and therefore
\[ \|A_t\|_{\mathrm{op}}^{-2} \geq (1 + 3 \sqrt{tn})^{-2} \geq 1 - 6 \sqrt{tn}. \]
Using the bound $N_t \leq \dim \Sym_n \leq n^2$, we obtain
\[
g'(t)  \leq  - \frac{1}{2} \E \left[ {\bf 1}_{G_t}  N_t \right] + 3 \sqrt{t} n^{5/2}.
\]
 
Let $\mathcal{E}_t \coloneqq  \mathcal{E}_{A_t}$ be the ellipsoid associated with the matrix $A_t$. By Lemma~\ref{lemme:CL},
\[ V(A_t,\mathscr{C}_\Lambda) = \{ x \otimes x \, : \, x \in \Lambda^* \cap \partial\mathcal{E}_t\}^\perp \]
and therefore (the factor $\frac{1}{2}$ coming from $(-x) \otimes (-x) = x \otimes x$),
\[ N_t  \geq \dim \Sym_n - \frac{1}{2} |\Lambda^* \cap \partial\mathcal{E}_t| .\]
We therefore have
\begin{align*}
 g'(t) & \leq  - \frac{1}{2} \dim \Sym_n \times \P(G_t) + \frac{1}{4} \E \left[ {\bf 1}_{G_t} |\Lambda^* \cap \partial\mathcal{E}_t| \right] + 3 \sqrt{t} n^{5/2} \\
 & \leq -\frac{n^2}{4} + \frac{1}{4} \E \left[ {\bf 1}_{G_t} |\Lambda^* \cap \partial\mathcal{E}_t| \right] + 3 \sqrt{t} n^{5/2} + n^2e^{-n/2} .
\end{align*}
 
If the event $G_t$ holds, then $A_t \succcurlyeq (1-3 \sqrt{tn}) \Id$ and consequently every $x \in \Lambda^*$ satisfies $\scalar{A_tx}{x} \geq (1-3\sqrt{tn}) \|x\|^2$. Thus the condition $x \in \partial \mathcal{E}_t$, equivalent to $\scalar{A_tx}{x}=1$, can only be satisfied if $x \in D_t$. We therefore have
\[ {\bf 1}_{G_t} |\Lambda^* \cap \partial\mathcal{E}_t|
= \sum_{\mathclap{x \in \Lambda^* \cap D_t}} {\bf 1}_{G_t} {\bf 1}_{\{
x \in \partial\mathcal{E}_t\}}
\]
and by linearity of expectation,
\begin{equation} \label{eq:sum-over-lattice-points}
\E \left[ {\bf 1}_{G_t} |\Lambda^* \cap \partial\mathcal{E}_t| \right]
\leq \sum_{\mathclap{x \in \Lambda^* \cap D_t}} \P(x \in \partial\mathcal{E}_t).
\end{equation}
 
We next use the following observation: for every $t>0$ and $x \in \Lambda^*$,
\[ x \in  \partial\mathcal{E}_t \iff \exists s \in [0,t] \, : \, \scalar{A_sx}{x} \leq 1 .\]
This equivalence follows from two facts. On one hand,
if $s \leq t$ then $V(A_t,\mathscr{C}_\Lambda) \subset V(A_s,\mathscr{C}_\Lambda)$ by Proposition~\ref{prop:exploration}(iii), so the equality $\scalar{A_sx}{x} = 1$ implies $\scalar{A_tx}{x}=1$. On the other hand, since $A_s \in \mathscr{C}_\Lambda$, saying $\scalar{A_sx}{x} \leq 1$ is equivalent to $\scalar{A_sx}{x} = 1$.
 
This allows us to write, using the coupling inequality proved in Proposition~\ref{prop:couplage} and setting $y \coloneqq \frac{x}{\|x\|}$,
\begin{align*} \P(x \in \partial\mathcal{E}_t) & =
 \P \Bigl( \inf_{0 \leq s \leq t} \scalar{A_sx}{x} \leq 1 \Bigr) \\
& \leq  \P \Bigl( \inf_{0 \leq s \leq t} \scalar{(\Id+W_s)x}{x} \leq 1 \Bigr) \\
& =  \P \Bigl( \inf_{0 \leq s \leq t} \scalar{W_sy}{y} \leq \frac{1}{\|x\|^2 } - 1 \Bigr).
\end{align*}
The process $\beta_s \coloneqq  \scalar{W_sy}{y}$ is a scalar Brownian motion. A consequence of the reflection principle \cite[Theorem 2.4]{LeGall13} is the following formula, for every $a>0$:
\[ \P \Bigl( \inf_{0 \leq s \leq t} \beta_s \leq -a \Bigr) =
\P \Bigl( \sup_{0 \leq s \leq t} \beta_s \geq a \Bigr) =
2 \P(|\beta_t| \geq a) =  2 \Phi \Bigl( \frac{a}{\sqrt{t}} \Bigr).
\]
Applied with $a=1-\|x\|^{-2}$, this formula yields the inequality
\[ \P(x \in \partial\mathcal{E}_t) \leq
 2 \Phi \Bigl( \frac{1}{\sqrt{t}}\Bigl( 1 - \frac{1}{\|x\|^2}\Bigr)\Bigr) \eqqcolon f_t(x).
\]
We can therefore deduce from~\eqref{eq:sum-over-lattice-points} that
\[ \E \left[ {\bf 1}_{G_t} |\Lambda^* \cap \partial\mathcal{E}_t| \right]
\leq \sum_{\mathclap{x \in \Lambda^* \cap D_t}} f_t(x) \eqqcolon K_t(\Lambda) \]
which completes the proof of Proposition~\ref{prop1}.
 
\subsection{Proof of Proposition~\ref{prop2}}
 
We may assume $n$ is large enough that $3\sqrt{nt_{\max}} \leq \frac{1}{2}$. Let $\Lambda \subset \R^n$ be a random lattice of covolume $\omega_n$.
Apply Siegel's formula \eqref{eq:siegel} to the function $f_t {\bf 1}_{D_t}$ to write
\[ \E \left[ K_t(\Lambda) \right]  = \frac{1}{\omega_n} \int_{D_t} f_t(x) \, \mathrm{d}x = \frac{2}{\omega_n} \int_{D_t} \Phi \Bigl( \frac{1}{\sqrt{t}} \Bigl(1 - \frac{1}{\|x\|^2}\Bigr) \Bigr) \]
After integration in polar coordinates, we obtain
\[ \E \left[ K_t(\Lambda) \right]
= 2n \int_0^{(1-3 \sqrt{nt})^{-1/2}} r^{n-1} \Phi \Bigl( \frac{1 - r^{-2}}{\sqrt{t}} \Bigr) \, \mathrm{d}r.
\]
Making the substitution $y=\frac{1-r^{-2}}{\sqrt{t}}$, we get
 \[ \E \left[ K_t(\Lambda) \right] = n \sqrt{t} \int_{-\iy}^{3 \sqrt{n}} \Phi(y) ( 1 - y \sqrt{t})^{-\frac{n+2}{2}} \, \mathrm{d}y. \]
An integration by parts then yields
\[
\E \left[ K_t(\Lambda) \right] = 2 \Phi(3\sqrt{n})(1-3\sqrt{nt})^{-n/2} + 2 \int_{-\iy}^{3 \sqrt{n}} e^{-y^2/2} (1-y \sqrt{t})^{-n/2} \, \frac{\mathrm{d}y}{\sqrt{2\pi}} .\]
 Using the inequality $\frac{1}{1-x} \leq \exp(2x)$ valid for $x \in [0,\frac{1}{2}]$ and the inequality $\Phi(x) \leq \exp(-x^2/2)$ valid for all $x \geq 0$, we get
\[ \Phi(3\sqrt{n})(1-3\sqrt{nt})^{-n/2} \leq \exp \bigl(- 4.5n + 3 n^{3/2} \sqrt{t} \bigr). \]
Since $t \leq t_{\max} = \frac{20 \log n}{n^2}$, this quantity tends to $0$ as $n \to \infty$.
 
We next use the inequality $\frac{1}{1-x} \leq \exp(x+x^2)$ valid for all $x \leq \frac{1}{2}$ to obtain
\begin{align*}
 \int_{-\iy}^{3 \sqrt{n}} e^{-y^2/2} (1-y \sqrt{t})^{-n/2} \,
 \frac{\mathrm{d}y}{\sqrt{2\pi}} & \leq \int_{-\iy}^{3 \sqrt{n}} \exp \Bigl(-\frac{y^2}{2} +  \frac{ny \sqrt{t}}{2} + \frac{ny^2t}{2} \Bigr) \, \frac{\mathrm{d}y}{\sqrt{2\pi}} \\
 & \leq \int_{-\iy}^{\iy} \exp \Bigl(-\frac{y^2}{2(1-nt)^{-1}} +  \frac{ny \sqrt{t}}{2} \Bigr) \, \frac{\mathrm{d}y}{\sqrt{2\pi}} \\
 & = \frac{1}{\sqrt{1-nt}} \exp \Bigl( \frac{n^2 t}{8 (1-nt)} \Bigr) \\
 & = \frac{\exp \bigl( \frac{n^3t^2}{8(1-nt)} \bigr)}{\sqrt{1-nt}}
 \times \exp \Bigl( \frac{n^2t}{8} \Bigr).
\end{align*}
This computation used the Laplace transform of Gaussian densities: for $\sigma=(1-nt)^{-1/2}$ and $\alpha  = n\sqrt{t}/2$,
\[
\int_{-\iy}^{\iy} \exp \Bigl( -\frac{y^2}{2\sigma^2} \Bigr) \exp(\alpha y) \, \frac{\mathrm{d}y}{\sqrt{2\pi}}= \sigma \exp \Bigl( \frac{\alpha^2\sigma^2}{2} \Bigr).
\]
 
Under the hypothesis $t \leq t_{\max} = \frac{20 \log n}{n^2}$, the factor in front of $\exp(n^2t/8)$ tends to $1$ as~$n \to \infty$, which completes the proof of Proposition~\ref{prop2}.
 
\subsubsection*{Acknowledgments}
 
I thank the participants of the Lyon probability reading group for all the talks and discussions around sphere packings, and in particular Mikael de la Salle for his valuable comments on this text.
 
\bibliographystyle{apalike}
\bibliography{ref} 

@misc{AGZ26,
      title={Stochastically evolving ellipsoids with symmetries}, 
      author={Elisha B. Abuya and Nihar Gargava and Yufei Zhao},
      year={2026},
      note={arXiv 2606.05105},
      archivePrefix={arXiv},
      primaryClass={math.MG},
      url={https://arxiv.org/abs/2606.05105}, 
}

@article {CT62,
    AUTHOR = {Ciesielski, Z. and Taylor, S. J.},
     TITLE = {First passage times and sojourn times for {B}rownian motion in
              space and the exact {H}ausdorff measure of the sample path},
   JOURNAL = {Trans. Amer. Math. Soc.},
  FJOURNAL = {Transactions of the American Mathematical Society},
    VOLUME = {103},
      YEAR = {1962},
     PAGES = {434--450},
      ISSN = {0002-9947,1088-6850},
   MRCLASS = {60.62},
  MRNUMBER = {143257},
MRREVIEWER = {D.\ A.\ Darling},
       DOI = {10.2307/1993838},
       URL = {https://doi.org/10.2307/1993838},
}

@book {LeGall13,
    AUTHOR = {Le Gall, Jean-Fran{\c c}ois},
     TITLE = {Mouvement brownien, martingales et calcul stochastique},
    SERIES = {Math\'ematiques \& Applications (Berlin) [Mathematics \&
              Applications]},
    VOLUME = {71},
 PUBLISHER = {Springer, Heidelberg},
      YEAR = {2013},
     PAGES = {viii+176},
      ISBN = {978-3-642-31897-9; 978-3-642-31898-6},
   MRCLASS = {60H05 (60G07 60G44 60H10 60J25)},
  MRNUMBER = {3184878},
MRREVIEWER = {Josep\ Vives},
       DOI = {10.1007/978-3-642-31898-6},
       URL = {https://doi.org/10.1007/978-3-642-31898-6},
}

@book{AS17,
 author = {Aubrun, Guillaume and Szarek, Stanis{\l}aw J.},
 title = {Alice and {Bob} meet {Banach}. {The} interface of asymptotic geometric analysis and quantum information theory},
 fseries = {Mathematical Surveys and Monographs},
 series = {Math. Surv. Monogr.},
 issn = {0076-5376},
 volume = {223},
 isbn = {978-1-4704-3468-7; 978-1-4704-4172-2},
 year = {2017},
 publisher = {Providence, RI: American Mathematical Society (AMS)},
 language = {English},
 doi = {10.1090/surv/223},
 zbMATH = {6780344},
 Zbl = {1402.46001}
}

@article {Hlawka43,
    AUTHOR = {Hlawka, Edmund},
     TITLE = {Zur {G}eometrie der {Z}ahlen},
   JOURNAL = {Math. Z.},
  FJOURNAL = {Mathematische Zeitschrift},
    VOLUME = {49},
      YEAR = {1943},
     PAGES = {285--312},
      ISSN = {0025-5874,1432-1823},
   MRCLASS = {10.0X},
  MRNUMBER = {9782},
MRREVIEWER = {K.\ Mahler},
       DOI = {10.1007/BF01174201},
       URL = {https://doi.org/10.1007/BF01174201},
}

@incollection {Benoist09,
    AUTHOR = {Benoist, Yves},
     TITLE = {Five lectures on lattices in semisimple {L}ie groups},
 BOOKTITLE = {G\'eom\'etries \`a{} courbure n\'egative ou nulle, groupes
              discrets et rigidit\'es},
    SERIES = {S\'emin. Congr.},
    VOLUME = {18},
     PAGES = {117--176},
 PUBLISHER = {Soc. Math. France, Paris},
      YEAR = {2009},
      ISBN = {978-2-85629-240-2},
   MRCLASS = {22E40 (11F06 20H10)},
  MRNUMBER = {2655311},
MRREVIEWER = {Dave\ Witte\ Morris},
}

@article{Siegel45,
 author = {Siegel, Carl Ludwig},
 title = {A mean value theorem in geometry of numbers},
 fjournal = {Annals of Mathematics. Second Series},
 journal = {Ann. Math. (2)},
 issn = {0003-486X},
 volume = {46},
 pages = {340--347},
 year = {1945},
 language = {English},
 doi = {10.2307/1969027},
 zbMATH = {3107230},
 Zbl = {0063.07011}
}

@article{SZ24,
 author = {Sardari, Naser Talebizadeh and Zargar, Masoud},
 title = {New upper bounds for spherical codes and packings},
 fjournal = {Mathematische Annalen},
 journal = {Math. Ann.},
 issn = {0025-5831},
 volume = {389},
 number = {4},
 pages = {3653--3703},
 year = {2024},
 language = {English},
 doi = {10.1007/s00208-023-02738-z},
 zbMATH = {7892675}
}

@article{CZ14,
 author = {Cohn, Henry and Zhao, Yufei},
 title = {Sphere packing bounds via spherical codes},
 fjournal = {Duke Mathematical Journal},
 journal = {Duke Math. J.},
 issn = {0012-7094},
 volume = {163},
 number = {10},
 pages = {1965--2002},
 year = {2014},
 language = {English},
 doi = {10.1215/00127094-2738857},
 zbMATH = {6331281},
 Zbl = {1296.05046}
}

@article{KL78,
 author = {Kabatyanskii, Grigori A. and Levenshtein, Vladimir I.},
 title = {On bounds to packings on the sphere and in space},
 fjournal = {Problemy Peredachi Informatsii},
 journal = {Probl. Peredachi Inf.},
 issn = {0555-2923},
 volume = {14},
 number = {1},
 pages = {3--25},
 year = {1978},
 language = {Russian},
 zbMATH = {3633251},
 Zbl = {0407.52005}
}

@article{Blichfeldt29,
 author = {Blichfeldt, Hans F.},
 title = {The minimum value of quadratic forms, and the closest packing of spheres.},
 fjournal = {Mathematische Annalen},
 journal = {Math. Ann.},
 issn = {0025-5831},
 volume = {101},
 pages = {605--608},
 year = {1929},
 language = {English},
 doi = {10.1007/BF01454863},
 url = {https://eudml.org/doc/159357},
 zbMATH = {2572953},
 JFM = {55.0721.01}
}

@article{Vance11,
 author = {Vance, Stephanie},
 title = {Improved sphere packing lower bounds from {Hurwitz} lattices},
 fjournal = {Advances in Mathematics},
 journal = {Adv. Math.},
 issn = {0001-8708},
 volume = {227},
 number = {5},
 pages = {2144--2156},
 year = {2011},
 language = {English},
 doi = {10.1016/j.aim.2011.04.016},
 zbMATH = {5918260},
 Zbl = {1228.11095}
}

@article{Ball92,
 author = {Ball, Keith},
 title = {A lower bound for the optimal density of lattice packings},
 fjournal = {IMRN. International Mathematics Research Notices},
 journal = {Int. Math. Res. Not.},
 issn = {1073-7928},
 volume = {1992},
 number = {10},
 pages = {217--221},
 year = {1992},
 language = {English},
 doi = {10.1155/S1073792892000242},
 zbMATH = {179210},
 Zbl = {0776.52006}
}

@article{DR47,
 author = {Davenport, Harold and Rogers, Claude A.},
 title = {Hlawka's theorem in the geometry of numbers},
 fjournal = {Duke Mathematical Journal},
 journal = {Duke Math. J.},
 issn = {0012-7094},
 volume = {14},
 pages = {367--375},
 year = {1947},
 language = {English},
 doi = {10.1215/S0012-7094-47-01429-4},
 zbMATH = {3047575},
 Zbl = {0030.34602}
}

@article{CKMRV17,
 author = {Cohn, Henry and Kumar, Abhinav and Miller, Stephen D. and Radchenko, Danylo and Viazovska, Maryna},
 title = {The sphere packing problem in dimension {{\(24\)}}},
 fjournal = {Annals of Mathematics. Second Series},
 journal = {Ann. Math. (2)},
 issn = {0003-486X},
 volume = {185},
 number = {3},
 pages = {1017--1033},
 year = {2017},
 language = {English},
 doi = {10.4007/annals.2017.185.3.8},
 zbMATH = {6731864},
 Zbl = {1370.52037}
}

@article{Viazovska17,
 author = {Viazovska, Maryna},
 title = {The sphere packing problem in dimension 8},
 fjournal = {Annals of Mathematics. Second Series},
 journal = {Ann. Math. (2)},
 issn = {0003-486X},
 volume = {185},
 number = {3},
 pages = {991--1015},
 year = {2017},
 language = {English},
 doi = {10.4007/annals.2017.185.3.7},
 zbMATH = {6731863},
 Zbl = {1373.52025}
}

@article{Hales05,
 author = {Hales, Thomas C.},
 title = {A proof of the {Kepler} conjecture},
 fjournal = {Annals of Mathematics. Second Series},
 journal = {Ann. Math. (2)},
 issn = {0003-486X},
 volume = {162},
 number = {3},
 pages = {1065--1185},
 year = {2005},
 language = {English},
 doi = {10.4007/annals.2005.162.1065},
 zbMATH = {5033806},
 Zbl = {1096.52010}
}

@article{Klartag26,
  title={Lattice packing of spheres in high dimensions using a stochastically evolving ellipsoid},
  author={Klartag, Boaz},
  journal={Inventiones mathematicae},
  pages={1--29},
  year={2026},
  publisher={Springer}
}

@article {Rogers47,
    AUTHOR = {Rogers, Claude A.},
     TITLE = {Existence theorems in the geometry of numbers},
   JOURNAL = {Ann. of Math. (2)},
  FJOURNAL = {Annals of Mathematics. Second Series},
    VOLUME = {48},
      YEAR = {1947},
     PAGES = {994--1002},
      ISSN = {0003-486X},
   MRCLASS = {10.0X},
  MRNUMBER = {22863},
MRREVIEWER = {V.\ Knichal},
       DOI = {10.2307/1969390},
       URL = {https://doi.org/10.2307/1969390},
}

@article {Venkatesh13,
    AUTHOR = {Venkatesh, Akshay},
     TITLE = {A note on sphere packings in high dimension},
   JOURNAL = {Int. Math. Res. Not. IMRN},
  FJOURNAL = {International Mathematics Research Notices. IMRN},
      YEAR = {2013},
    NUMBER = {7},
     PAGES = {1628--1642},
      ISSN = {1073-7928,1687-0247},
   MRCLASS = {52C17},
  MRNUMBER = {3044452},
MRREVIEWER = {Ranjeet\ Kaur\ Sehmi},
       DOI = {10.1093/imrn/rns096},
       URL = {https://doi.org/10.1093/imrn/rns096},
}

@misc{CJMS23,
  title={A new lower bound for sphere packing},
  author={Campos, Marcelo and Jenssen, Matthew and Michelen, Marcus and Sahasrabudhe, Julian},
  note={arXiv 2302.09412},
  archivePrefix={arXiv},
  primaryClass={math.MG},
  YEAR = {2023},
}

@misc{SW25,
  title={The lattice packing problem in dimension 9 by {V}oronoi's algorithm},
  author={Dutour Sikiri{\'c}, Mathieu and van Woerden, Wessel},
  note={arXiv 2508.20719},
  archivePrefix={arXiv},
  primaryClass={math.NT},
  year={2025}
}

@book{Splag,
    AUTHOR = {Conway, John H. and Sloane, Neil J. A.},
     TITLE = {Sphere packings, lattices and groups},
    SERIES = {Grundlehren der mathematischen Wissenschaften [Fundamental  Principles of Mathematical Sciences]},
    VOLUME = {290},
   EDITION = {Third},

 PUBLISHER = {Springer-Verlag, New York},
      YEAR = {1999},
     PAGES = {lxxiv+703},
      ISBN = {0-387-98585-9},
   MRCLASS = {11H31 (05B40 11H06 20D08 52C07 52C17 94B75 94C30)},
  MRNUMBER = {1662447},
MRREVIEWER = {Renaud\ Coulangeon},
       DOI = {10.1007/978-1-4757-6568-7},
       URL = {https://doi.org/10.1007/978-1-4757-6568-7},
}
 
\end{document}